\documentclass[reqno]{amsart}
\usepackage{amsmath, amsthm, amssymb, amstext}

\usepackage[left=3cm,right=3cm,top=3cm,bottom=3cm]{geometry}
\usepackage{hyperref,xcolor}
\hypersetup{pdfborder={0 0 0},colorlinks}
\usepackage{enumitem}
\allowdisplaybreaks
\usepackage{todonotes}

\newtheorem{theorem}{Theorem}
\newtheorem{remark}[theorem]{Remark}
\newtheorem{lemma}[theorem]{Lemma}
\newtheorem{proposition}[theorem]{Proposition}

\newtheorem{definition}[theorem]{Definition}
\newtheorem{example}[theorem]{Example}

\newcommand*\diff{\mathrm{d}}

\DeclareMathOperator*{\divergenz}{div}              %
\DeclareMathOperator*{\Ss}{S}

\newcommand{\N}{\mathbb{N}}

\newcommand{\R}{\mathbb{R}}

\newcommand{\Lp}[1]{L^{#1}(\Omega)}

\newcommand{\Lprand}[1]{L^{#1}(\partial\Omega)}
\newcommand{\Wp}[1]{W^{1,#1}(\Omega)}

\newcommand{\Wpzero}[1]{W^{1,#1}_0(\Omega)}

\newcommand{\lan}{\langle}
\newcommand{\ran}{\rangle}
\newcommand{\eps}{\varepsilon}
\newcommand{\ph}{\varphi}

\newcommand{\into}{\int_{\Omega}}

\newcommand{\weak}{\rightharpoonup}

\newcommand{\close}{\overline{\Omega}}

\renewcommand{\l}{\left}
\renewcommand{\r}{\right}
\numberwithin{theorem}{section}
\numberwithin{equation}{section}

\title[The sub-supersolution method for variable exponent double phase systems]{The sub-supersolution method for variable exponent double phase systems with nonlinear boundary conditions}

\author[U.\,Guarnotta]{Umberto Guarnotta}
\address[U.\,Guarnotta]{Department of Mathematics and Computer Science, University of Palermo, 90123 Palermo, Italy}
\email{umberto.guarnotta@unipa.it}

\author[R.\,Livrea]{Roberto Livrea}
\address[R.\,Livrea]{Department of Mathematics and Computer Science, University of Palermo, 90123 Palermo, Italy}
\email{roberto.livrea@unipa.it}

\author[P.\,Winkert]{Patrick Winkert}
\address[P.\,Winkert]{Technische Universit\"{a}t Berlin, Institut f\"{u}r Mathematik, Stra\ss e des 17.\,Juni 136, 10623 Berlin, Germany}
\email{winkert@math.tu-berlin.de}

\dedicatory{Dedicated to Professor Siegfried Carl on the occasion of his 70th birthday}

\begin{document}

\begin{abstract}
	In this paper we study quasilinear elliptic systems driven by variable exponent double phase operators involving fully coupled right-hand sides and nonlinear boundary conditions. The aim of our work is to establish an enclosure and existence result for such systems by means of trapping regions formed by pairs of sup- and supersolutions. Under very general assumptions on the data we then apply our result to get infinitely many solutions. Moreover, we also discuss the case when we have homogeneous Dirichlet boundary conditions and present some existence results for this kind of problem.
\end{abstract}

\subjclass{}
\keywords{}

\maketitle

\section{Introduction}

In this paper we consider the following variable exponent double phase system with nonlinear boundary conditions
\begin{equation}\label{problem}
	\left\{
	\begin{aligned}
		-\divergenz \big(|\nabla u_1|^{p_1(x)-2}\nabla u_1 +\mu_1(x)|\nabla u_1|^{q_1(x)-2}\nabla u_1 \big)&=  f_1(x,u_1,u_2,\nabla u_1,\nabla u_2)\quad&& \text{in } \Omega,\\
		-\divergenz \big(|\nabla u_2|^{p_2(x)-2}\nabla u_2 +\mu_2(x)|\nabla u_2|^{q_2(x)-2}\nabla u_2 \big)&=  f_2(x,u_1,u_2,\nabla u_1,\nabla u_2)\quad&& \text{in } \Omega,\\
			\l(|\nabla u_1|^{p_1(x)-2}\nabla u_1 +\mu_1(x)|\nabla u_1|^{q_1(x)-2}\nabla u_1\r)\cdot \nu&= g_1(x,u_1,u_2) &&\text{on } \partial\Omega,\\
		\l(|\nabla u_2|^{p_2(x)-2}\nabla u_2 +\mu_2(x)|\nabla u_2|^{q_2(x)-2}\nabla u_2\r)\cdot \nu&= g_2(x,u_1,u_2) &&\text{on } \partial\Omega,
	\end{aligned}
	\right.
\end{equation}
where $\Omega \subseteq \R^N$, $N\geq 2$, is a bounded domain with Lipschitz boundary $\partial\Omega$, $\nu(x)$ denotes the unit normal of $\Omega$ at the point $x\in\partial\Omega$, $f_i\colon\Omega \times \R\times\R\times \R^N\times\R^N\to\R$ and $g_i\colon\partial\Omega\times \R\times \R \to\R$ are Carath\'eodory functions for $i=1,2$ that satisfy local growth conditions (see hypotheses \eqref{H2}) and we suppose the following assumptions on the exponents and the weight functions: \begin{enumerate}[label=\textnormal{(H$1$)},ref=\textnormal{H$1$}]
	\item\label{H1}
	$p_i,q_i\in C(\close)$ such that $1<p_i(x)<N$ and $p_i(x) < q_i(x)<p_i^*(x)$ for all $x\in\close$, as well as $0 \leq \mu_i(\cdot) \in \Lp{\infty}$, where $p_i^*$ is given by
	\begin{align*}
		p_i^*(x)= \frac{Np_i(x)}{N-p_i(x)}\quad\text{for }x\in \close,
	\end{align*}
	for $i=1,2$.
\end{enumerate}
The operator in \eqref{problem} is the so-called variable exponent double phase operator given by
\begin{align*}
	\divergenz \big(|\nabla u_i|^{p_i(x)-2}\nabla u_i +\mu_i(x)|\nabla u_i|^{q_i(x)-2}\nabla u_i \big), \quad u\in W^{1,\mathcal{H}_i}(\Omega),
\end{align*}
defined in a suitable Musielak-Orlicz Sobolev space $W^{1,\mathcal{H}_i}(\Omega)$, $i=1,2$, which has been recently studied in Crespo-Blanco-Gasi\'nski-Harjulehto-Winkert \cite{Crespo-Blanco-Gasinski-Harjulehto-Winkert-2022}. The study of such operators goes back to Zhikov \cite{Zhikov-1986}  who introduced for the first time energy functionals  defined by
\begin{align*}
	\omega \mapsto \int_\Omega \big(|\nabla  \omega|^p+a(x)|\nabla  \omega|^q\big)\,\diff x.
\end{align*}
Such functionals have been used to describe models for strongly anisotropic materials in the context of homogenization and elasticity. It also has several mathematical applications in the study of duality theory and of the Lavrentiev
gap phenomenon; see Zhikov \cite{Zhikov-1995,Zhikov-2011}.

The main objective of our paper is to establish a method of sub- and supersolution in terms of trapping region of the system \eqref{problem} under very general local structure conditions on the nonlinearities involved. As an application, we present some existence results to the system \eqref{problem}  under very mild and easily verifiable conditions on the data. In addition, we will also study the corresponding Dirichlet system and get a sub-supersolution approach including some existence results. The novelty of our paper is the combination of the variable exponent double phase operator with fully coupled convective right-hand sides along with coupled nonlinear boundary functions. To the best of our knowledge, such general systems have not been treated in the literature, even if we replace our operator with the $p_i$-Laplacian, that is, $\mu_i\equiv 0$ for $i=1,2$. 

Our paper is motivated by the work of Carl-Motreanu \cite{Carl-Motreanu-2017} who studied the elliptic system
\begin{align*}
	-\Delta_{p_i}u_i=f_i(x,u_1,u_2,\nabla u_1, \nabla u_2) \quad\text{in }\Omega, \quad u_i=0 \quad\text{on }\partial\Omega,
\end{align*}
where they obtain extremal positive and negative solutions of the system by combining the theory of pseudomonotone operators, regularity results as well as a strong maximum principle. On the contrary, in the present paper we obtain existence and multiplicity results by using neither regularity theory nor strong maximum principle, which are not available in our setting. The method of sub- and supersolution is a very powerful tool and has been used in several works: here we mention, for example, the papers of Carl-Le-Winkert \cite{Carl-Le-Winkert-2022}, Carl-Winkert \cite{Carl-Winkert-2009}, Motreanu-Sciammetta-Tornatore \cite{Motreanu-Sciammetta-Tornatore-2020}; see also the monographs of Carl-Le \cite{Carl-Le-2021} and Carl-Le-Motreanu \cite{Carl-Le-Motreanu-2007}.

We also point out that the right-hand sides in \eqref{problem} depend on the gradients of the solutions. Such reactions are said to be convection terms. The difficulty in the study of such terms is their nonvariational character, that is, the standard variational tools cannot be applied, even in the scalar case (i.e., for a single differential equation). For systems with convection terms only few works are available: we mention the papers of Guarnotta-Marano \cite{Guarnotta-Marano-2021-a, Guarnotta-Marano-2021-b}, Guarnotta-Marano-Moussaoui \cite{Guarnotta-Marano-Moussaoui-2022} and Faria-Miyagaki-Pereira \cite{Faria-Miyagaki-Pereira-2014}. Neumann systems without gradient dependence on the nonlinearity can be found in Chabrowski \cite{Chabrowski-2011} and de Godoi-Miyagaki-Rodrigues \cite{deGodoi-Miyagaki-Rodrigues-2016}. Finally, we mention some works pertaining equations exhibiting convection terms and subjected to Dirichlet or Neumann boundary conditions: we refer to Averna-Motreanu-Tornatore \cite{Averna-Motreanu-Tornatore-2016}, de Araujo-Faria \cite{de-Araujo-Faria-2019}, Dupaigne-Ghergu-R\u{a}dulescu \cite{Dupaigne-Ghergu-Radulescu-2007}, El Manouni-Marino-Winkert \cite{El-Manouni-Marino-Winkert-2022}, Faraci-Motreanu-Puglisi \cite{Faraci-Motreanu-Puglisi-2015}, Faraci-Puglisi \cite{Faraci-Puglisi-2016},  Figueiredo-Madeira \cite{Figueiredo-Madeira-2021}, Gasi\'nski-Papageorgiou \cite{Gasinski-Papageorgiou-2017}, Gasi\'nski-Winkert \cite{Gasinski-Winkert-2020}, Guarnotta-Marano-Motreanu \cite{Guarnotta-Marano-Motreanu-2020}, Liu-Motreanu-Zeng \cite{Liu-Motreanu-Zeng-2019}, Marano-Winkert \cite{Marano-Winkert-2019}, Motreanu-Tornatore \cite{Motreanu-Tornatore-2017}, Motreanu-Winkert \cite{Motreanu-Winkert-2019}, Papageorgiou-R\u{a}dulescu-Repov\v{s} \cite{Papageorgiou-Radulescu-Repovs-2020}, and Vetro-Winkert \cite{Vetro-Winkert-2022}.

The paper is organized as follows. In Section \ref{Section2} we present the main preliminaries, including the properties of the Musielak-Orlicz Sobolev space, the double phase operator and the definition of trapping region (see Definition \ref{def-sub-supersolution}). Section \ref{Section3} is devoted to our abstract existence result for given pairs of sub-supersolution (see Theorem \ref{theorem-sub-supersolution}), while in Section \ref{Section4} we present several existence results with a construction of sub-supersolution (see Theorems \ref{Neumannsol} and \ref{Neumannsols}). Finally, in Section \ref{Section5} we consider the corresponding Dirichlet systems including the method of sub-supersolution and some existence results (see Theorems \ref{theorem-sub-supersolution-dirichlet} and \ref{Dirichletsol}).

\section{Preliminaries}\label{Section2}

In this section we recall some facts about variable exponent Lebesgue spaces, Musielak-Orlicz Sobolev spaces, and properties of the variable exponent double phase operator. We refer to the books of Diening-Harjulehto-H\"{a}st\"{o}-R$\mathring{\text{u}}$\v{z}i\v{c}ka \cite{Diening-Harjulehto-Hasto-Ruzicka-2011} and Harjulehto-H\"{a}st\"{o} \cite{Harjulehto-Hasto-2019}; see also the papers of Crespo-Blanco-Gasi\'nski-Harjulehto-Winkert \cite{Crespo-Blanco-Gasinski-Harjulehto-Winkert-2022}, Fan-Zhao \cite{Fan-Zhao-2001}, and Kov{\'a}{\v{c}}ik-R{\'a}kosn{\'{\i}}k \cite{Kovacik-Rakosnik-1991}.

Let $\Omega$ be a bounded domain in $\mathbb{R}^N$ with Lipschitz boundary $\partial\Omega$ and let
\begin{align*}
	C_+(\close):=\big\{h \in C(\close) \, : \, 1<h(x) \text{ for all }x\in \close\big\}.
\end{align*}
For any $r\in C_+(\close)$ we define
\begin{align*}
	r^-=\min_{x\in \close}r(x) \quad\text{and}\quad r^+=\max_{x\in\close} r(x).
\end{align*}
Let $M(\Omega)$ be the space of all measurable functions $u\colon \Omega\to\R$. For a given $r \in C_+(\close)$, the variable exponent Lebesgue space $\Lp{r(\cdot)}$ is defined as
\begin{align*}
	\Lp{r(\cdot)}=\l\{u \in M(\Omega)\,:\, \into |u|^{r(x)}\,\diff x<\infty \r\}
\end{align*}
equipped with the Luxemburg norm given by
\begin{align*}
	\|u\|_{r(\cdot)} =\inf \l \{\lambda>0 \, : \, \into \l(\frac{|u|}{\lambda}\r)^{r(x)}\,\diff x \leq 1 \r\}.
\end{align*}
We know that $(\Lp{r(\cdot)},\|\cdot\|_{r(\cdot)})$ is a separable and reflexive Banach space. Similarly we introduce the variable exponent boundary Lebesgue space $(\Lprand{r(\cdot)},\|\cdot\|_{r(\cdot),\partial\Omega})$ by using the $(N-1)$-dimensional Hausdorff surface measure $\sigma$.  

Let $r' \in C_+(\close)$ be the conjugate variable exponent to $r$, that is,
\begin{align*}
	\frac{1}{r(x)}+\frac{1}{r'(x)}=1 \quad\text{for all }x\in\close.
\end{align*}
We have that $\Lp{r(\cdot)}^*=\Lp{r'(\cdot)}$ and H\"older's inequality 
holds true, namely
\begin{align*}
	\into |uv| \diff x \leq \l[\frac{1}{r^-}+\frac{1}{(r')^-}\r] \|u\|_{r(\cdot)}\|v\|_{r'(\cdot)} \leq 2 \|u\|_{r(\cdot)}\|v\|_{r'(\cdot)}
\end{align*}
for all $u\in \Lp{r(\cdot)}$ and for all $v \in \Lp{r'(\cdot)}$. Furthermore, if $r_1, r_2\in C_+(\close)$ and $r_1(x) \leq r_2(x)$ for all $x\in \close$, then we have the continuous embedding
\begin{align*}
	\Lp{r_2(\cdot)} \hookrightarrow \Lp{r_1(\cdot)}.
\end{align*}

Next, we are going to introduce Musielak-Orlicz Lebesgue and Sobolev spaces. To this end, suppose hypotheses \eqref{H1} and for $i=1,2$ let $\mathcal{H}_i\colon \Omega \times [0,\infty)\to [0,\infty)$ be the nonlinear function defined by
\begin{align*}
	\mathcal H_i(x,t)= t^{p_i(x)}+\mu(x)t^{q_i(x)}.
\end{align*}
The Musielak-Orlicz space $\Lp{\mathcal{H}_i}$ is defined by
\begin{align*}
	\Lp{\mathcal{H}_i}
	=\left \{u\in M(\Omega)\,:\, \rho_{\mathcal{H}_i}(u)<+\infty \right \}
\end{align*}
equipped with the Luxemburg norm
\begin{align*}
	\|u\|_{\mathcal{H}_i} = \inf \left \{ \tau >0 \,:\, \rho_{\mathcal{H}_i}\left(\frac{u}{\tau}\right) \leq 1  \right \},
\end{align*}
where the modular function $\rho_{\mathcal{H}_i}$ is given by
\begin{align*}
	\rho_{\mathcal{H}_i}(u):=\into \mathcal{H}_i(x,|u|)\,\diff x=\into \big(|u|^{p_i(x)}+\mu_i(x)|u|^{q_i(x)}\big)\,\diff x.
\end{align*}

We have the following relation between the norm $\|\cdot\|_{\mathcal{H}_i}$ and the modular $\rho_{\mathcal{H}_i}$ (see Crespo-Blanco-Gasi\'nski-Harjulehto-Winkert \cite[Proposition 2.13]{Crespo-Blanco-Gasinski-Harjulehto-Winkert-2022}).

\begin{proposition}\label{proposition_modular_properties}
	Let hypotheses \eqref{H1} be satisfied. For $i=1,2$ we have the following assertions.
	\begin{enumerate}
		\item[\textnormal{(i)}]
		If $u\neq 0$, then $\|u\|_{\mathcal{H}_i}=\lambda$ if and only if $ \rho_{\mathcal{H}_i}(\frac{u}{\lambda})=1$.
		\item[\textnormal{(ii)}]
		$\|u\|_{\mathcal{H}_i}<1$ (resp.\,$>1$, $=1$) if and only if $ \rho_{\mathcal{H}_i}(u)<1$ (resp.\,$>1$, $=1$).
		\item[\textnormal{(iii)}]
		If $\|u\|_{\mathcal{H}_i}<1$, then $\|u\|_{\mathcal{H}_i}^{q_i^+}\leqslant \rho_{\mathcal{H}_i}(u)\leqslant\|u\|_{\mathcal{H}_i}^{p_i^-}$.
		\item[\textnormal{(iv)}]
		If $\|u\|_{\mathcal{H}_i}>1$, then $\|u\|_{\mathcal{H}_i}^{p_i^-}\leqslant \rho_{\mathcal{H}_i}(u)\leqslant\|u\|_{\mathcal{H}_i}^{q_i^+}$.
		\item[\textnormal{(v)}]
		$\|u\|_{\mathcal{H}_i}\to 0$ if and only if $ \rho_{\mathcal{H}_i}(u)\to 0$.
		\item[\textnormal{(vi)}]
		$\|u\|_{\mathcal{H}_i}\to +\infty$ if and only if $ \rho_{\mathcal{H}_i}(u)\to +\infty$.
		\item[\textnormal{(vii)}]
		$\|u\|_{\mathcal{H}_i}\to 1$ if and only if $ \rho_{\mathcal{H}_i}(u)\to 1$.
		\item[\textnormal{(viii)}]
		If $u_n \to u$ in $\Lp{\mathcal{H}_i}$, then $\rho_{\mathcal{H}_i} (u_n) \to \rho_{\mathcal{H}_i} (u)$.
	\end{enumerate}
\end{proposition}

The Musielak-Orlicz Sobolev space $W^{1,\mathcal{H}_i}(\Omega)$ is defined by
\begin{align*}
	W^{1,\mathcal{H}_i}(\Omega)= \left \{u \in L^{\mathcal{H}_i}(\Omega) \,:\, |\nabla u| \in L^{\mathcal{H}_i}(\Omega) \right\}
\end{align*}
equipped with the norm
\begin{align*}
	\|u\|_{1,\mathcal{H}_i}= \|\nabla u \|_{\mathcal{H}_i}+\|u\|_{\mathcal{H}_i},
\end{align*}
where $\|\nabla u\|_{\mathcal{H}_i}=\|\,|\nabla u|\,\|_{\mathcal{H}_i}$ and $i=1,2$. The completion of $C^\infty_0(\Omega)$ in $W^{1,\mathcal{H}_i}(\Omega)$ is denoted by $W^{1,\mathcal{H}_i}_0(\Omega)$. We know that $\Lp{\mathcal{H}_i}$, $\Wpzero{\mathcal{H}_i}$, $\Wp{\mathcal{H}_i}$ are reflexive Banach spaces (see Crespo-Blanco-Gasi\'nski-Harjulehto-Winkert \cite[Proposition 2.12]{Crespo-Blanco-Gasinski-Harjulehto-Winkert-2022}). 

Next, we recall some embedding results for the spaces $\Lp{\mathcal{H}_i}$, $\Wpzero{\mathcal{H}_i}$, $\Wp{\mathcal{H}_i}$ (see Crespo-Blanco-Gasi\'nski-Harjulehto-Winkert \cite[Proposition 2.16]{Crespo-Blanco-Gasinski-Harjulehto-Winkert-2022}).

\begin{proposition}\label{proposition_embeddings}
	Let hypotheses \eqref{H1} be satisfied and let
	\begin{align*}
		p_i^*(x):=\frac{Np_i(x)}{N-p_i(x)} \quad \text{and}\quad (p_i)_*(x):=\frac{(N-1)p_i(x)}{N-p_i(x)}
		\quad\text{for all }x\in\close
	\end{align*}
	be the critical exponents to $p_i$ for $i=1,2$. Then the following embeddings hold for $i=1,2$:
	\begin{enumerate}
		\item[\textnormal{(i)}]
			$\Lp{\mathcal{H}_i} \hookrightarrow \Lp{r_i(\cdot)}$, $\Wp{\mathcal{H}_i}\hookrightarrow \Wp{r_i(\cdot)}$, $\Wpzero{\mathcal{H}_i}\hookrightarrow \Wpzero{r_i(\cdot)}$
		are continuous for all $r\in C(\close)$ with $1\leq r(x)\leq p_i(x)$ for all $x \in \Omega$;
		\item[\textnormal{(ii)}]
			$\Wp{\mathcal{H}_i} \hookrightarrow \Lp{r_i(\cdot)}$ and $\Wpzero{\mathcal{H}_i} \hookrightarrow \Lp{r_i(\cdot)}$ are compact for $r_i \in C(\close) $ with $ 1 \leq r_i(x) < p_i^*(x)$ for all $x\in \close$;
		\item[\textnormal{(iii)}]
			$\Wp{\mathcal{H}_i} \hookrightarrow \Lprand{r_i(\cdot)}$ and $\Wpzero{\mathcal{H}_i} \hookrightarrow \Lprand{r_i(\cdot)}$ are compact for $r_i \in C(\close) $ with $ 1 \leq r_i(x) < (p_i)_*(x)$ for all $x\in \close$;
		\item[\textnormal{(iv)}]
			$L^{q_i(\cdot)}(\Omega) \hookrightarrow \Lp{\mathcal{H}_i}$ is continuous.
	\end{enumerate}
\end{proposition}

For $i=1,2$, let $A_i\colon \Wp{\mathcal{H}_i}  \to \Wp{\mathcal{H}_i}^*$ be defined by
\begin{align}\label{operator_representation}
	\langle A_i(u_i),v_i\rangle_{\mathcal{H}_i} :=\into \big(|\nabla u_i|^{p_i(x)-2}\nabla u_i+\mu_i(x)|\nabla u_i|^{q_i(x)-2}\nabla u_i \big)\cdot\nabla v_i \,\diff x
\end{align}
for all $u_i,v_i\in\Wp{\mathcal{H}_i} $, where $\lan\,\cdot\,,\,\cdot\,\ran_{\mathcal{H}_i}$ is the duality pairing between $\Wp{\mathcal{H}_i}$ and its dual space $\Wp{\mathcal{H}_i}^*$.  The operator $A_i\colon \Wp{\mathcal{H}_i}\to \Wp{\mathcal{H}_i}^*$ has the following properties (see Crespo-Blanco-Gasi\'nski-Harjulehto-Winkert \cite[Proposition 3.4]{Crespo-Blanco-Gasinski-Harjulehto-Winkert-2022}).

\begin{proposition}\label{properties_operator_double_phase}
	Let hypotheses \eqref{H1}  be satisfied. Then, the operators $A_i$ defined in \eqref{operator_representation} are bounded, continuous, strictly monotone and of type $(\Ss_+)$, that is,
	\begin{align*}
		u_n\weak u \quad \text{in }\Wp{\mathcal{H}_i} \quad\text{and}\quad  \limsup_{n\to\infty}\,\langle A_iu_n,u_n-u\rangle\le 0
	\end{align*}
	imply $u_n\to u$ in $\Wp{\mathcal{H}_i}$.
\end{proposition}

Next, we define the product spaces
\begin{align*}
	\mathcal{L}&:=\Lp{\mathcal{H}_1}\times \Lp{\mathcal{H}_2},\\
	L^{p_1(\cdot),p_2(\cdot)}(\Omega)&:=\Lp{p_1(\cdot)}\times \Lp{p_2(\cdot)},\\
	L^{q_1(\cdot),q_2(\cdot)}(\Omega)&:=\Lp{q_1(\cdot)}\times \Lp{q_2(\cdot)},\\
	L^{p_1(\cdot),p_2(\cdot)}(\partial\Omega)&:=\Lprand{p_1(\cdot)}\times \Lprand{p_2(\cdot)},\\
	\mathcal{W}&:=\Wp{\mathcal{H}_1}\times \Wp{\mathcal{H}_2}\\
\end{align*}
equipped with the norms
\begin{align*}
	\|u\|_{\mathcal{L}}&=\|u\|_{\mathcal{H}_1}+\|u\|_{\mathcal{H}_2},\\
	\|u\|_{L^{p_1(\cdot),p_2(\cdot)}(\Omega)}&=\|u\|_{p_1(\cdot)}+\|u\|_{p_2(\cdot)},\\
	\|u\|_{L^{q_1(\cdot),q_2(\cdot)}(\Omega)}&=\|u\|_{q_1(\cdot)}+\|u\|_{q_2(\cdot)},\\
	\|u\|_{L^{p_1(\cdot),p_2(\cdot)}(\partial\Omega)}&=\|u\|_{p_1(\cdot),\partial\Omega}+\|u\|_{p_2(\cdot),\partial\Omega},\\
	\|u\|_{\mathcal{W}}&=\|u\|_{1,\mathcal{H}_1}+\|u\|_{1,\mathcal{H}_2},
\end{align*}
respectively. Based on Proposition \ref{proposition_embeddings} we have the compact embeddings
\begin{align}\label{compacness-embedding}
	\mathcal{W}\hookrightarrow\mathcal{L}, \quad \mathcal{W}\hookrightarrow L^{p_1(\cdot),p_2(\cdot)}(\Omega), \quad \mathcal{W}\hookrightarrow L^{q_1(\cdot),q_2(\cdot)}(\Omega), \quad \mathcal{W}\hookrightarrow L^{p_1(\cdot),p_2(\cdot)}(\partial\Omega).
\end{align}

\begin{definition}
	We say that $(u_1,u_2)\in\mathcal{W}$ is a weak solution of problem \eqref{problem} if
	\begin{align}\label{defsol1}
		\begin{split}
			&\into \l(|\nabla u_1|^{p_1(x)-2}\nabla u_1+\mu_1(x)|\nabla u_1|^{q_1(x)-2}\nabla u_1\r)\cdot \nabla v_1\,\diff x
			-\int_{\partial\Omega} g_1(x,u_1,u_2)v_1\,\diff \sigma\\
			&=\into f_1(x,u_1,u_2,\nabla u_1,\nabla u_2)v_1\,\diff x\\
		\end{split}
	\end{align}
	and
	\begin{align}\label{defsol2}
		\begin{split}
			&\into \l(|\nabla u_2|^{p_2(x)-2}\nabla u_2+\mu_2(x)|\nabla u_2|^{q_2(x)-2}\nabla u_2\r)\cdot \nabla v_2\,\diff x
			-\int_{\partial\Omega} g_2(x,u_1,u_2)v_2\,\diff \sigma\\
			&=\into f_2(x,u_1,u_2,\nabla u_1,\nabla u_2)v_2\,\diff x\\
		\end{split}
	\end{align}
	hold true for all $(v_1,v_2)\in\mathcal{W}$ and all the integrals in \eqref{defsol1} and \eqref{defsol2} are finite.
\end{definition}

Next, we introduce the notion of weak sub- and supersolution to \eqref{problem}.

\begin{definition}\label{def-sub-supersolution}
	We say that $(\underline{u}_1,\underline{u}_2)$, $(\overline{u}_1, \overline{u}_2)\in \mathcal{W}$ form a pair of sub- and supersolution of problem \eqref{problem} if $\underline{u}_i\leq \overline{u}_i$ a.\,e.\,in $\Omega$ for $i=1,2$ and
	\begin{align}\label{defsubsuper1}
		\begin{split}
			&\into \l(|\nabla \underline{u}_1|^{p_1(x)-2}\nabla \underline{u}_1+\mu_1(x)|\nabla \underline{u}_1|^{q_1(x)-2}\nabla \underline{u}_1\r)\cdot \nabla v_1\,\diff x
			-\into f_1(x,\underline{u}_1,w_2,\nabla \underline{u}_1,\nabla w_2)v_1\,\diff x\\
			&-\int_{\partial\Omega} g_1(x,\underline{u}_1,w_2)v_1\,\diff \sigma\\
			&+\into \l(|\nabla \underline{u}_2|^{p_2(x)-2}\nabla \underline{u}_2+\mu_2(x)|\nabla \underline{u}_2|^{q_2(x)-2}\nabla \underline{u}_2\r)\cdot \nabla v_2\,\diff x
			-\into f_2(x,w_1,\underline{u}_2,\nabla w_1,\nabla \underline{u}_2)v_2\,\diff x\\
			&-\int_{\partial\Omega} g_2(x,w_1,\underline{u}_2)v_2\,\diff \sigma \leq 0
		\end{split}
	\end{align}
	and
	\begin{align}\label{defsubsuper2}
		\begin{split}
			&\into \l(|\nabla \overline{u}_1|^{p_1(x)-2}\nabla \overline{u}_1+\mu_1(x)|\nabla \overline{u}_1|^{q_1(x)-2}\nabla \overline{u}_1\r)\cdot \nabla v_1\,\diff x
			-\into f_1(x,\overline{u}_1,w_2,\nabla \overline{u}_1,\nabla w_2)v_1\,\diff x\\
			&-\int_{\partial\Omega} g_2(x,\overline{u}_1,w_2)v_1\,\diff \sigma\\
			&+\into \l(|\nabla \overline{u}_2|^{p_2(x)-2}\nabla \overline{u}_2+\mu_2(x)|\nabla \overline{u}_2|^{q_2(x)-2}\nabla \overline{u}_2\r)\cdot \nabla v_2\,\diff x
			-\into f_2(x,w_1,\overline{u}_2,\nabla w_1,\nabla \overline{u}_2)v_2\,\diff x\\
			&-\int_{\partial\Omega} g_2(x,w_1,\overline{u}_2)v_2\,\diff \sigma \geq 0
		\end{split}
	\end{align}
	for all $(v_1,v_2)\in\mathcal{W}$, $v_1,v_2\geq 0$ a.\,e.\,in $\Omega$ and for all $(w_1,w_2)\in \mathcal{W}$ such that $\underline{u}_i\leq w_i\leq \overline{u}_i$  for $i=1,2$ and with all integrals in \eqref{defsubsuper1} and \eqref{defsubsuper2} to be finite.

	If $\underline{u}=(\underline{u}_1,\underline{u}_2)$, $\overline{u}=(\overline{u}_1,\overline{u}_2)$ is a pair of sub- and supersolution, then the order interval $[\underline{u},\overline{u}]=[\underline{u}_1,\overline{u}_1] \times [\underline{u}_2,\overline{u}_2]$ is called trapping region, where
	\begin{align*}
		\l[\underline{u}_i,\overline{u}_i\r]
		=\big\{u\in\Wp{\mathcal{H}_i}\,:\,\underline{u}_i\leq u\leq \overline{u}_i\text{ a.\,e.\,in }\Omega\big\}.
	\end{align*}
\end{definition}

We now recall some definitions that we will use in the sequel (see Carl-Le-Motreanu \cite[Definitions 2.95 and 2.96]{Carl-Le-Motreanu-2007}).

\begin{definition}\label{SplusPM}
	Let $X$ be a reflexive Banach space, $X^*$ its dual space, and denote by $\langle \cdot \,, \cdot\rangle$ its duality pairing. Let $A\colon X\to X^*$. Then $A$ is called
	\begin{enumerate}[leftmargin=1cm]
		\item[\textnormal{(i)}]
			coercive if
			\begin{align*}
				\lim_{\|u\|_X \to \infty} \frac{\langle Au, u \rangle}{\|u\|_X} = +\infty.
			\end{align*}
		\item[\textnormal{(ii)}]
			completely continuous if  $u_n \weak u$ in $X$ implies $Au_n\to Au$ in $X^*$.
		\item[\textnormal{(iii)}]
			demicontinuous if  $u_n \to u$ in $X$ implies $Au_n\weak Au$ in $X^*$.
		\item[\textnormal{(iv)}]
			pseudomonotone if 
			\begin{align*}
				u_n \weak u \quad \text{in } X
				\quad\text{and}\quad 
				\limsup_{n\to \infty}\, \langle Au_n,u_n-u\rangle \leq 0
			\end{align*}
			imply 
			\begin{align*}
				\liminf_{n \to \infty}\, \langle Au_n,u_n-v\rangle \geq \langle Au , u-v\rangle \quad\text{for all }v \in X.
			\end{align*}
		\item[\textnormal{(v)}]
			to satisfy the $(\Ss_+)$-property if 
			\begin{align*}
				u_n \weak u \quad\text{in }X 
				\quad\text{and}\quad \limsup_{n\to \infty}\, \langle Au_n,u_n-u\rangle \leq 0 
			\end{align*}
			imply $u_n\to u$ in $X$.
	\end{enumerate}
\end{definition}

\begin{remark}\label{comp-compcont}
In the context of Definition \ref{SplusPM}, any completely continuous operator is compact and any linear compact operator is completely continuous (see Zeidler \cite[Proposition 26.2]{Zeidler-1990}).
\end{remark}

The following helpful lemma can be found, for example, in Franc$\mathring{\text{u}}$ \cite[Lemma 6.7]{Francu-1990} (see also Zeidler \cite[Proposition 27.6]{Zeidler-1990}).

\begin{lemma}\label{lemma-pseudomonotone}
	Let $X$ be a reflexive Banach space and let $A\colon X\to X^*$ be a demicontinuous operator satisfying the $(\Ss_+)$-property. Then $A$ is pseudomonotone.
\end{lemma}

The next lemma is recently obtained in Gambera-Guarnotta \cite[Lemma 2.2]{Gambera-Guarnotta-2022}.

\begin{lemma}\label{lemma-splus}
	Let $X$ be a Banach space, $A\colon X\to X^*$ be of type $(\Ss_+)$, and $B\colon X\to X^*$ be compact. Then $A+B$ is of type $(\Ss_+)$ as well. 
\end{lemma}

We are going to apply the following surjectivity result for pseudomonotone operators (see, for example, Papageorgiou-Winkert \cite[Theorem 6.1.57]{Papageorgiou-Winkert-2018}).

\begin{theorem}\label{theorem_pseudomonotone}
	Let $X$ be a real, reflexive Banach space, let $A\colon X\to X^*$ be a pseudomonotone, bounded, and coercive operator, and $b\in X^*$. Then, a solution of the equation $Au=b$ exists.
\end{theorem}

Finally, for any $s \in \R$ we denote $s_\pm = \max \{ \pm s, 0 \}$, that means $s = s_+ - s_-$ and $|s| = s_+ + s_-$. For any function $v \colon \Omega \to \R$ we denote $v_\pm (\cdot) = [v(\cdot)]_\pm$.

\section{Sub-Supersolution approach}\label{Section3}

In this section we are going to prove a sub- and supersolution existence result for the system \eqref{problem} under very general structure conditions on the data. 

Let $\underline{u}=(\underline{u}_1,\underline{u}_2)$, $\overline{u}=(\overline{u}_1,\overline{u}_2)$ be a pair of sub- and supersolution of problem \eqref{problem} in the sense of Definition \ref{def-sub-supersolution}. We suppose the following assumptions.

\begin{enumerate}[label=\textnormal{(H$2$)},ref=\textnormal{H$2$}]
	\item\label{H2}
	For $i=1,2$ the functions $f_i\colon\Omega \times \R\times\R\times \R^N\times\R^N\to\R$ and $g_i\colon\partial\Omega\times \R\times \R \to\R$ are Carath\'eodory functions  satisfying the following conditions:
	\begin{enumerate}[itemsep=0.2cm,label=\textnormal{(\roman*)},ref=\textnormal{(H$2$)(\roman*)}]
		\item\label{H2i}
		there exist $\ph_i\in \Lp{p_i'(\cdot)}$ and $c_i>0$ such that
		\begin{align*}
			|f_1(x,s_1,s_2,\xi_1,\xi_2)| \leq \ph_1(x)+c_1\l(|\xi_1|^{p_1(x)-1}+|\xi_2|^{\frac{p_2(x)}{p_1'(x)}}\r),\\
			|f_2(x,s_1,s_2,\xi_1,\xi_2)| \leq \ph_2(x)+c_2\l(|\xi_1|^{\frac{p_1(x)}{p_2'(x)}}+|\xi_2|^{p_2(x)-1}\r),
		\end{align*}
		for a.\,a.\,$x\in\Omega$, for all $s=(s_1,s_2)\in [\underline{u}(x),\overline{u}(x)]$, and for all $\xi_i\in\R^N$;
		\item\label{H2ii}
		there exist $\psi_i\in \Lprand{p_i'(\cdot)}$ such that
		\begin{align*}
			|g_1(x,s_1,s_2)| \leq \psi_1(x)
			\quad\text{and}\quad
			|g_2(x,s_1,s_2)| \leq \psi_2(x),
		\end{align*}
		for a.\,a.\,$x\in\Omega$ and for all $s=(s_1,s_2)\in [\underline{u}(x),\overline{u}(x)]$.
	\end{enumerate}
\end{enumerate}

\begin{remark}
	Under hypotheses  \eqref{H1} and \eqref{H2}, for any $(u_1,u_2) \in \mathcal{W}\cap[\underline{u},\overline{u}]$, all the integrals appearing in \eqref{defsol1} and  \eqref{defsol2} are finite.
\end{remark}

Our main theorem in this section reads as follows.

\begin{theorem}\label{theorem-sub-supersolution}
	Let hypotheses \eqref{H1} and \eqref{H2} be satisfied. If $[\underline{u},\overline{u}]$ is a trapping region of \eqref{problem}, then the system in \eqref{problem} has a solution $u\in\mathcal{W}\cap [\underline{u},\overline{u}]$. 
\end{theorem}

\begin{proof}
	We split the proof into three steps.
	
	{\bf Step 1:} Preliminaries
	
	First, we introduce truncation operators $T_k\colon \Wp{\mathcal{H}_k}\to\Wp{\mathcal{H}_k}$ for $k=1,2$ defined by
	\begin{align}\label{truncation-Tk}
		T_k(u_k)(x)=
		\begin{cases}
			\overline{u}_k(x)&\text{if }u_k(x)>\overline{u}_k(x),\\
			u_k(x) &\text{if }\underline{u}_k(x)\leq u_k(x)\leq \overline{u}_k(x),\\
			\underline{u}_k(x)&\text{if }u_k(x)<\underline{u}_k(x).
		\end{cases}
	\end{align}
	We know that $T_k\colon \Wp{\mathcal{H}_k}\to\Wp{\mathcal{H}_k}$ are continuous and bounded. Next, we introduce the cut-off functions $b_k\colon \Omega\times\R\to\R$ for $k=1,2$ defined by
	\begin{align}\label{def-b}
		b_k(x,s)=
		\begin{cases}
			(s-\overline{u}_k(x))^{q_k(x)-1}&\text{if }s>\overline{u}_k(x),\\
			0 &\text{if }\underline{u}_k(x)\leq s\leq \overline{u}_k(x),\\
			-(\underline{u}_k(x)-s)^{q_k(x)-1}&\text{if }s<\underline{u}_k(x).
		\end{cases}
	\end{align}
	It is clear that $b_k$ are Carath\'eodory functions fulfilling the growth
	\begin{align}\label{growth-b}
		|b_k(x,s)|\leq \hat{\ph}_k(x)+\hat{c}_k|s|^{q_k(x)-1}
	\end{align}
	for a.\,a.\,$x\in\Omega$ and for all $s\in\R$, where $\hat{\ph}_k\in \Lp{q_i'(\cdot)}$ and $\hat{c}_k>0$. Moreover, we have the following estimate:
	\begin{align}\label{estimate-b}
		\into b_k(x,u_k)u_k\,\diff x
		\geq \hat{a}_k \into |u_k|^{q_k(x)}\,\diff x- \hat{b}_k
	\end{align}
	for all $u \in \Lp{q_k(\cdot)}$, where $\hat{a}_k,\hat{b}_k$ are some positive constants. From the growth condition in \eqref{growth-b} we know that the corresponding Nemytskij operators $B_k\colon \Lp{q_k(\cdot)}\to \Lp{q_k'(\cdot)}$, defined by $B_k(u_k)(x)=b_k(x,u_k(x))$, are well defined, bounded, and continuous for $k=1,2$. Hence, the operator $\mathcal{B}u=(B_1(u_1),B_2(u_2))$ is also well defined. By virtue of \eqref{compacness-embedding} and Remark \ref{comp-compcont}, we know that $\mathcal{B}\colon\mathcal{W}\hookrightarrow L^{q_1(\cdot),q_2(\cdot)}(\Omega)\to L^{q_1(\cdot),q_2(\cdot)}(\Omega)^*\hookrightarrow\mathcal{W}^*$ is bounded and completely continuous.
	
	Let $\lambda=(\lambda_1,\lambda_2)$ with $\lambda_k\geq 0$ and set
	\begin{align*}
		\lambda\mathcal{B}(u)=(\lambda_1 B_1(u_1),\lambda_2 B_2(u_2)).
	\end{align*}
	Furthermore, we set
	\begin{align*}
		\mathcal{F}(u)=\l(F_1(T_1 u_1,T_2 u_2,\nabla (T_1 u_1), \nabla (T_2 u_2)),F_2(T_1 u_1,T_2 u_2,\nabla (T_1 u_1),\nabla (T_2 u_2))\r),
	\end{align*}
	where $F_k$ denote the Nemytskij operators related to $f_k$, which are well defined for $k=1,2$ since the ranges of $T_1,T_2$ lie within the trapping region $[\underline{u},\overline{u}]$. Therefore, due to the growth condition in \ref{H2i} and the compact embedding $\mathcal{W}\hookrightarrow L^{p_1(\cdot),p_2(\cdot)}(\Omega)$ (see \eqref{compacness-embedding}), we have that
	\begin{align*}
		\mathcal{F}\colon \mathcal{W}\to L^{p_1(\cdot),p_2(\cdot)}(\Omega)^*\hookrightarrow \mathcal{W}^*
	\end{align*}
	is bounded and compact. For the boundary term, we define 
	\begin{align*}
		\mathcal{G}(u)=(G_1(T_1 u_1,T_2 u_2), G_2(T_1 u_1,T_2 u_2)),
	\end{align*}
	where $G_k$ are the Nemytskij operators generated by $g_k$. We know that
	\begin{align*}
		\mathcal{G}(u): \mathcal{W}\hookrightarrow L^{p_1(\cdot),p_2(\cdot)}(\partial\Omega)\to L^{p_1(\cdot),p_2(\cdot)}(\partial\Omega)^*\hookrightarrow\mathcal{W}^*
	\end{align*}
	is well defined, completely continuous, and bounded, due to \ref{H2i}, the compactness of the trace operator (see \eqref{compacness-embedding}), and Remark \ref{comp-compcont}.
	
	Finally, let $\mathcal{A}(u)=(A_1(u_1),A_2(u_2))$ where $A_k$ are defined in \eqref{operator_representation}. Because of Proposition \ref{properties_operator_double_phase}, it is clear that $\mathcal{A}\colon\mathcal{W}\to\mathcal{W}^*$ is bounded, continuous, strictly monotone, and of type $(\Ss_+)$. We have the representations
	\begin{align*}
		\langle \mathcal{A}(u),v\rangle_{\mathcal{W}}
		&=\sum_{k=1}^2 \into \big(|\nabla u_k|^{p_k(x)-2}\nabla u_k+\mu_k(x)|\nabla u_k|^{q_k(x)-2}\nabla u_k\big) \cdot \nabla v_k\,\diff x\\
		\langle \mathcal{B}(u),v\rangle_{\mathcal{W}}
		&=\sum_{k=1}^2 \into B_k(u_1,u_2,\nabla u_1, \nabla u_2)v_k\,\diff x\\
		\langle \mathcal{F}(u),v\rangle_{\mathcal{W}}
		&=\sum_{k=1}^2 \into F_k(T_1 u_1,T_2 u_2,\nabla (T_1 u_1), \nabla (T_2 u_2))v_k\,\diff x\\
		\langle \mathcal{G}(u),v\rangle_{\mathcal{W}}
		&=\sum_{k=1}^2 \int_{\partial\Omega} G_k(T_1 u_1,T_2 u_2)v_k\,\diff \sigma
	\end{align*}
	for all $u,v\in\mathcal{W}$.
	
	Using the notations above, $u\in\mathcal{W}\cap[\underline{u},\overline{u}]$ is a solution to \eqref{problem} if and only if
	\begin{align*}
		\langle \mathcal{A}(u),v\rangle_{\mathcal{W}}=
		\langle \mathcal{F}(u),v\rangle_{\mathcal{W}} 
		+\langle \mathcal{G}(u),v\rangle_{\mathcal{W}}
		\quad \text{for all }v\in \mathcal{W}.
	\end{align*}

	{\bf Step 2:} Auxiliary problem
	
	Let $\mathcal{T}(u)=(T_1u_1,T_2u_2)$, where $T_k$ are the truncation operators defined in \eqref{truncation-Tk}. Now we consider the following auxiliary problem given in the form
	\begin{align*}
		u\in\mathcal{W}\ : \
		\langle \mathcal{A}(u)+\lambda\mathcal{B}(u),v\rangle_{\mathcal{W}}=
		\langle \mathcal{F}(u),v\rangle_{\mathcal{W}} 
		+\langle \mathcal{G}(u),v\rangle_{\mathcal{W}}
		\quad \text{for all }v\in\mathcal{W}, 
	\end{align*}
	where $\lambda=(\lambda_1,\lambda_2)$ with $\lambda_k\geq 0$ to be specified later. Let $\Phi\colon\mathcal{W}\to\mathcal{W}^*$ be given by
	\begin{align*}
		\Phi(u):=\mathcal{A}(u)+\lambda\mathcal{B}(u)-\mathcal{F}(u)-\mathcal{G}(u).
	\end{align*}
	First, we know that $\Phi$ is bounded and continuous. Since $\mathcal{A}$ is of type $(\Ss_+)$ (see Proposition \ref{properties_operator_double_phase}) and $\mathcal{B},\mathcal{F},\mathcal{G}$ are compact (and hence completely continuous; see Remark \ref{comp-compcont}), we can apply Lemma \ref{lemma-splus} to get that $\Phi$ is of type $(\Ss_+)$ as well. Lemma \ref{lemma-pseudomonotone} then implies that $\Phi$ is pseudomonotone.

	Next, we are going to show that $\Phi\colon \mathcal{W}\to \mathcal{W}^*$ is coercive. To this end, using hypothesis \ref{H2i} and Young's inequality we estimate
	\begin{align}\label{estF}
		\begin{split}
			|\langle\mathcal{F}(u),u\rangle_{\mathcal{W}}| &\leq \sum_{k=1}^2 \into |f_k(x,T_1u_1,T_2u_2,\nabla(T_1u_1),\nabla(T_2u_2))||u_k|\,\diff x \\
			&\leq \into \l[ \ph_1|u_1| + c_1\l( |\nabla(T_1u_1)|^{p_1(x)-1}|u_1|+|\nabla(T_2u_2)|^{\frac{p_2(x)}{p_1'(x)}}|u_1| \r)\r]\,\diff x \\
			&\quad +\into \l[ \ph_2|u_2| + c_2\l( |\nabla(T_1u_1)|^{\frac{p_1(x)}{p_2'(x)}}|u_2|+|\nabla(T_2u_2)|^{p_2(x)-1}|u_2| \r)\r]\,\diff x \\
			&\leq \sum_{k=1}^2 \l( \into \ph_k^{p_k'(x)}\,\diff x + \into |u_k|^{p_k(x)}\,\diff x \r) \\
			&\quad +2 \sum_{k=1}^2 \l( \eps \into |\nabla(T_ku_k)|^{p_k(x)}\,\diff x + C_{\eps,k} \into |u_k|^{p_k(x)}\,\diff x \r) \\
			&\leq \sum_{k=1}^2 \l( 2\eps \rho_{\mathcal{H}_k}(|\nabla u_k|) +(2C_{\eps,k}+1) \rho_{\mathcal{H}_k}(u_k) + C_k \r),
		\end{split}
	\end{align}
	for suitable $C_k$ depending on $\ph_k$ and $C_{\eps,k}>0$ depending on both $\eps$ and $p_k$. 
	
	Next, we consider the operator $\mathcal{G}$. To this end, for any $u_k \in L^{\mathcal{H}_k}(\Omega)$, we define
	\begin{align*}
		\eta_k(u_k) =
		\begin{cases}
			q_k^+ &\text{if }\|u_k\|_{\mathcal{H}_k}<1, \\
			p_k^- &\text{if }\|u_k\|_{\mathcal{H}_k}\geq 1, \\
		\end{cases}
	\end{align*}
	and observe that $\eta_k>1$ for all $u_k\in L^{\mathcal{H}_k}(\Omega)$. Using the definition of $\eta_k$ along with \ref{H2ii}, H\"older's inequality, the embedding inequality $\|u\|_{p_k(\cdot),\partial\Omega} \leq S_k\|u\|_{1,\mathcal{H}_k}$ (cf. Proposition \ref{proposition_embeddings}(iii)), and Young's inequality gives
	\begin{align}\label{estG}
		\begin{split}
			|\langle\mathcal{G}(u),u\rangle_{\mathcal{W}}| &\leq \sum_{k=1}^2 \int_{\partial\Omega}|g_k(x,T_ku_k)||u_k|\,\diff \sigma\\ 
			&\leq \sum_{k=1}^2 \int_{\partial\Omega} \psi_k|u_k|\,\diff \sigma \\
			&\leq 2\sum_{k=1}^2 \|\psi_k\|_{p_k'(\cdot),\partial\Omega}\|u_k\|_{p_k(\cdot),\partial\Omega} \\
			&\leq 2\sum_{k=1}^2 S_k\|\psi_k\|_{p_k'(\cdot),\partial\Omega}\|u_k\|_{1,\mathcal{H}_k} \\
			&= \sum_{k=1}^2 C_k \l( \|u_k\|_{\mathcal{H}_k}+\|\nabla u_k\|_{\mathcal{H}_k} \r) \\
			&\leq \sum_{k=1}^2 \l[ \eps \l( \|u_k\|_{\mathcal{H}_k}^{\eta_k(u_k)}+\|\nabla u_k\|_{\mathcal{H}_k}^{\eta_k(|\nabla u_k|)} \r) + C_{\eps,k} \l( C_k^{\eta_k'(u_k)}+C_k^{\eta_k'(|\nabla u_k|)} \r) \r] \\
			&\leq \sum_{k=1}^2 \l[ \eps\l(\rho_{\mathcal{H}_k}(u_k) + \rho_{\mathcal{H}_k}(|\nabla u_k|)\r) + \hat{C}_{\eps,k} \r]
		\end{split}
	\end{align}
	for suitable positive constants $C_k$ depending on $\psi_k$'s, while $C_{\eps,k},\hat{C}_{\eps,k}$ also depend on $\eps$. On the other hand,  we have
	\begin{align*}
		\into |u_k|^{q_k(x)}\,\diff x \geq \tilde{a}_k\rho_{\mathcal{H}_k}(u_k)-\tilde{b}_k
	\end{align*}
	for suitable $\tilde{a}_k,\tilde{b}_k>0$. Using this along with \eqref{estimate-b}  we get
	\begin{align}\label{estAB}
		\begin{split}
			\langle \mathcal{A}(u)+\lambda\mathcal{B}(u),u\rangle_{\mathcal{W}} 
			&\geq  \sum_{k=1}^2 \l( \rho_{\mathcal{H}_k}(|\nabla u_k|) + \lambda \hat{a}_k\into |u_k|^{q_k(x)}\,\diff x-\lambda\hat{b}_k \r) \\
			&\geq \sum_{k=1}^2 \l( \rho_{\mathcal{H}_k}(|\nabla u_k|) + \tilde{a}_k\hat{a}_k\lambda \rho_{\mathcal{H}_k}(u_k)-\tilde{b}_k\hat{a}_k\lambda-\hat{b}_k\lambda \r).
		\end{split}
	\end{align}
	Combining \eqref{estF}, \eqref{estG}, and \eqref{estAB} together we obtain
	\begin{align}\label{coercive-1}
		\begin{split}
			\langle\Phi(u),u\rangle_{\mathcal{W}} &\geq \sum_{k=1}^2 \l[ (1-3\eps)\rho_{\mathcal{H}_k}(|\nabla u_k|) + (\tilde{a}_k\hat{a}_k\lambda-2C_{\eps,k}-1-\eps) \rho_{\mathcal{H}_k}(u_k)\r.\\
			&\qquad\quad\l.-\tilde{b}_k\hat{a}_k\lambda-\hat{b}_k\lambda - C_k - \hat{C}_{\eps,k} \r].
		\end{split}
	\end{align}
	Now, we choose 
	\begin{align*}
		\eps<\frac{1}{3}
		\quad\text{and}\quad \lambda>\frac{2C_{\eps,k}+1+\eps}{\min_k \tilde{a}_k\hat{a}_k}
	\end{align*}
	in \eqref{coercive-1} and use the fact that $\rho_{\mathcal{H}_k}(u_k)+\rho_{\mathcal{H}_k}(|\nabla u_k|) \to \infty$ if and only if $\|u_k\|_{1,\mathcal{H}_k} \to \infty$ (see Proposition \ref{proposition_modular_properties}(vi)). Hence, we infer that $\langle\Phi(u),u\rangle_{\mathcal{W}} \to +\infty$ as $\|u\|_{\mathcal{W}} \to \infty$.
	
	Since $\Phi$ is bounded, continuous, pseudomonotone, and coercive, the main theorem on pseudomonotone operators (see Theorem \ref{theorem_pseudomonotone}) implies the existence of $u\in\mathcal{W}$ such that $\Phi(u)=0$. 
	
	{\bf Step 3:} Comparison
	
	It remains to prove that $u\in[\underline{u},\overline{u}]$. We set $(u-\overline{u})_+ = ((u_1-\overline{u}_1)_+,(u_2-\overline{u}_2)_+)$. From $\Phi(u)=0$, besides recalling the definitions of $\mathcal{F},\mathcal{G}$, we deduce
	\begin{align}\label{comparison-solution1}
		\begin{split}
			0 &= \langle A_1(u_1)+\lambda_1 B_1(u_1),(u_1-\overline{u}_1)_+\rangle_{\mathcal{H}_1} \\
			&\quad - \langle F_1(\overline{u}_1,T_2u_2,\nabla \overline{u}_1,\nabla (T_2u_2))+G_1(\overline{u}_1,T_2u_2),(u_1-\overline{u}_1)_+\rangle_{\mathcal{H}_1}
		\end{split}
	\end{align}
	and
	\begin{align}\label{comparison-solution2}
		\begin{split}
		0 &= \langle A_2(u_2)+\lambda_2 B_2(u_2),(u_2-\overline{u}_2)_+\rangle_{\mathcal{H}_2} \\
		&\quad - \langle F_2(T_1u_1,\overline{u}_2,\nabla (T_1u_1),\nabla \overline{u}_2)+G_2(T_1u_1,\overline{u}_2),(u_2-\overline{u}_2)_+\rangle_{\mathcal{H}_2}.
		\end{split}
	\end{align}
	On the other hand, since $\overline{u}$ is a supersolution to \eqref{problem}, it turns out that
	\begin{align}\label{comparsion-supersolution}
		\begin{split}
			&\sum_{k=1}^2 \langle A_k(\overline{u}_k),(u_k-\overline{u}_k)_+\rangle_{\mathcal{H}_k}\\ 
			&\geq \langle F_1(\overline{u}_1,T_2u_2,\nabla \overline{u}_1,\nabla (T_2u_2))+G_1(\overline{u}_1,T_2u_2),(u_1-\overline{u}_1)_+\rangle_{\mathcal{H}_1} \\
			&\quad +\langle F_2(T_1u_1,\overline{u}_2,\nabla (T_1u_1),\nabla \overline{u}_2)+G_2(T_1u_1,\overline{u}_2),(u_2-\overline{u}_2)_+\rangle_{\mathcal{H}_2}.
		\end{split}
	\end{align}
	Hence, from \eqref{comparison-solution1}, \eqref{comparison-solution2} and \eqref{comparsion-supersolution} along with the monotonicity of $A_k$ (see Proposition \ref{properties_operator_double_phase}), we obtain
	\begin{align}\label{comparison-monotonicity}
		\sum_{k=1}^2 \lambda_k \langle B_k(u_k),(u_k-\overline{u}_k)_+\rangle_{\mathcal{H}_k} \leq \sum_{k=1}^2 \langle A_k(\overline{u}_k)-A_k(u_k),(u_k-\overline{u}_k)_+\rangle_{\mathcal{H}_k} \leq 0.
	\end{align}
	According to the definition of $B_k$ (see \eqref{def-b}), \eqref{comparison-monotonicity} implies
	\begin{align*}
		\into (u_k-\overline{u}_k)_+^{q_k(x)}\,\diff x \leq 0.
	\end{align*}
	Thus, $u_k \leq \overline{u}_k$ a.\,e.\,in $\Omega$. Similarly, we show $\underline{u}_k \leq u_k$ a.\,e.\,in $\Omega$ by applying the definition of subsolution. Therefore, we have shown that $u \in [\underline{u},\overline{u}]$ and so, by the definition of the truncations in \eqref{truncation-Tk} and the functions $b_k$ in \eqref{def-b}, we see that $u\in \mathcal{W}$ turns out to be a weak solution of the system \eqref{problem} lying within $ [\underline{u},\overline{u}]$.
\end{proof}

\section{Sub- and supersolutions}\label{Section4}

This section is devoted to the construction of pairs of sub- and supersolution for the system \eqref{problem}. Following ideas of Guarnotta-Marano \cite{Guarnotta-Marano-2021-a} (see also the papers of D'Agu\`\i-Sciammetta \cite{DAgui-Sciammetta-2012} and Motreanu-Sciammetta-Tornatore \cite{Motreanu-Sciammetta-Tornatore-2020} for a single equation), we prove the existence of infinitely many solutions to \eqref{problem} under suitable sign conditions on the nonlinearities, exhibiting an oscillatory behavior. We suppose the following assumptions on the Carath\'eodory functions$f_i\colon\Omega \times \R\times\R\times\R^N\times \R^N\to\R$ and $g_i\colon\partial\Omega\times\R\times  \R\to\R$ for $i=1,2$.

\begin{enumerate}[label=\textnormal{(H$3$)},ref=\textnormal{H$3$}]
	\item\label{H3}
	There exist $h_i,k_i\in\R$ such that $h_i\leq k_i$ and
	\begin{align*}
		f_1(x,k_1,s_2,0,\xi_2) \leq &\,0 \leq f_1(x,h_1,s_2,0,\xi_2) && \text{for a.\,a.\,}x\in\Omega,\\
		g_1(x,k_1,s_2) \leq &\,0 \leq g_1(x,h_1,s_2)&& \text{for a.\,a.\,}x\in\partial\Omega,\\
		f_2(x,s_1,k_2,\xi_1,0) \leq &\,0 \leq f_2(x,s_1,h_2,\xi_1,0)&& \text{for a.\,a.\,}x\in\Omega,\\
		g_2(x,s_1,k_2) \leq &\,0 \leq g_2(x,s_1,h_2)&& \text{for a.\,a.\,}x\in\partial\Omega,
	\end{align*}
	for all $(s_1,s_2)\in[h_1,k_1]\times[h_2,k_2]$ and for all $\xi_i\in\R^N$.
\end{enumerate}

We have the following existence result.

\begin{theorem}
	\label{Neumannsol}
	Let hypotheses \eqref{H1} and \eqref{H3} be satisfied. Suppose that \eqref{H2} is fulfilled for a.\,a.\,$x \in \Omega$, for all $s_i\in[h_i,k_i]$, and for all $\xi_i\in\R^N$, $i=1,2$. Then there exists a weak solution $(u_1,u_2)\in \mathcal{W}$ of system \eqref{problem} satisfying $h_i\leq u_i\leq k_i$ for $i=1,2$.
\end{theorem}

\begin{proof}
	We set $\underline{u}_i:=h_i$ and $\overline{u}_i:=k_i$. By \eqref{H3} we have $\underline{u}_i\leq\overline{u}_i$. For all $(v_1,v_2)\in\mathcal{W}$ with $v_1,v_2\geq 0$ a.\,e.\,in $\Omega$ and for all $(w_1,w_2)\in \mathcal{W}$ such that $\underline{u}_i\leq w_i\leq \overline{u}_i$, we get
	\begin{align*}
		&\into \l(|\nabla \underline{u}_1|^{p_1(x)-2}\nabla \underline{u}_1+\mu_1(x)|\nabla \underline{u}_1|^{q_1(x)-2}\nabla \underline{u}_1\r)\cdot \nabla v_1\,\diff x
		-\into f_1(x,\underline{u}_1,w_2,\nabla \underline{u}_1,\nabla w_2)v_1\,\diff x\\
		&-\int_{\partial\Omega} g_1(x,\underline{u}_1,w_2)v_1\,\diff \sigma\\
		&+\into \l(|\nabla \underline{u}_2|^{p_2(x)-2}\nabla \underline{u}_2+\mu_2(x)|\nabla \underline{u}_2|^{q_2(x)-2}\nabla \underline{u}_2\r)\cdot \nabla v_2\,\diff x
		-\into f_2(x,w_1,\underline{u}_2,\nabla w_1,\nabla \underline{u}_2)v_2\,\diff x\\
		&-\int_{\partial\Omega} g_2(x,w_1,\underline{u}_2)v_2\,\diff \sigma \\
		=&-\into f_1(x,\underline{u}_1,w_2,\nabla \underline{u}_1,\nabla w_2)v_1\,\diff x-\int_{\partial\Omega} g_1(x,\underline{u}_1,w_2)v_1\,\diff \sigma\\
		&-\into f_2(x,w_1,\underline{u}_2,\nabla w_1,\nabla \underline{u}_2)v_2\,\diff x-\int_{\partial\Omega} g_2(x,w_1,\underline{u}_2)v_2\,\diff \sigma \leq 0.
	\end{align*}
	Analogous computations concerning $\overline{u}_1,\overline{u}_2$ prove that $(\underline{u}_1,\underline{u}_2)$ and $(\overline{u}_1,\overline{u}_2)$ form a pair of sub- and supersolution of problem \eqref{problem}. Then, Theorem \ref{theorem-sub-supersolution} implies the existence of a weak solution $(u_1,u_2) \in \mathcal{W}$ of \eqref{problem} satisfying $\underline{u}_i\leq u_i\leq \overline{u}_i$ for $i=1,2$.
\end{proof}

If we strengthen our assumptions, we can obtain more solutions. For this purpose, we assume the following hypothesis.

\begin{enumerate}[label=\textnormal{(H$4$)},ref=\textnormal{H$4$}]
	\item\label{H4}
	For all $n\in\N$, there exist $h_i^{(n)},k_i^{(n)}\in\R$ such that
	\begin{align*}
		\mbox{either} \quad h_i^{(n)}\leq k_i^{(n)}<h_i^{(n+1)} \quad \mbox{or} \quad k_i^{(n+1)}<h_i^{(n)}\leq k_i^{(n)}
	\end{align*}
	and
	\begin{align*}
		f_1\l(x,k_1^{(n)},s_2,0,\xi_2\r) \leq &\,0 \leq f_1\l(x,h_1^{(n)},s_2,0,\xi_2\r)&&\text{for a.\,a.\,}x\in\Omega,\\
		g_1\l(x,k_1^{(n)},s_2\r) \leq &\,0 \leq g_1\l(x,h_1^{(n)},s_2\r)&&\text{for a.\,a.\,}x\in\partial\Omega,\\
		f_2\l(x,s_1,k_2^{(n)},\xi_1,0\r) \leq &\,0 \leq f_2\l(x,s_1,h_2^{(n)},\xi_1,0\r)&&\text{for a.\,a.\,}x\in\Omega,\\
		g_2\l(x,s_1,k_2^{(n)}\r) \leq &\,0 \leq g_2\l(x,s_1,h_2^{(n)}\r)&&\text{for a.\,a.\,}x\in\partial\Omega,
	\end{align*}
	for all $(s_1,s_2)\in[h_1^{(n)},k_1^{(n)}]\times[h_2^{(n)},k_2^{(n)}]$, for all $\xi_i\in\R^N$, and for all $n\in\N$.
\end{enumerate}

\begin{theorem}
	\label{Neumannsols}
	Let hypotheses \eqref{H1} and \eqref{H4} be satisfied. Suppose that, for all $n\in\N$, \eqref{H2} is fulfilled  for a.\,a.\,$x \in \Omega$, for all $s_i\in[h_i^{(n)},k_i^{(n)}]$, and for all $\xi_i\in\R^N$. Then there exists a sequence $\{(u_1^{(n)},u_2^{(n)})\}\subseteq \mathcal{W}$ of pairwise distinct solutions to problem \eqref{problem}. Moreover, $u_i^{(n)}\leq u_i^{(n+1)}$ (resp., $u_i^{(n+1)}\leq u_i^{(n)}$) provided $k_i^{(n)}<h_i^{(n+1)}$ (resp., $k_i^{(n+1)}<h_i^{(n)}$) for all $n\in\N$.
\end{theorem}

\begin{proof}
	It suffices to apply Theorem \ref{Neumannsol} for all $n\in\N$, with $h_i=h_i^{(n)}$ and $k_i=k_i^{(n)}$, and observe that $u_i^{(n)}(x)\leq k_i^{(n)}<h_i^{(n+1)}\leq u_i^{(n+1)}(x)$ for a.\,a.\,$x\in\Omega$, provided $k_i^{(n)}<h_i^{(n+1)}$ (the other case works similarly).
\end{proof}

The following example satisfies hypotheses \eqref{H3} and \eqref{H4}.

\begin{example}
	Let $c_i>0$ and $\rho_i \in L^\infty(\Omega)$ be such that $|\rho_i(x)| \leq \frac{1}{2}$ a.\,e.\,in $\Omega$. Then the functions
	\begin{align*}
		f_1(x,s_1,s_2,\xi_1,\xi_2) &= \sin s_1 + \frac{1}{2}\cos s_2 + c_1|\xi_1|^{p_1(x)-1} + \frac{1}{\pi}\arctan |\xi_2| \\
		f_2(x,s_1,s_2,\xi_1,\xi_2) &= \frac{1}{2}\sin s_1 + \cos s_2 + \frac{1}{\pi}\arctan |\xi_1| + c_2|\xi_2|^{p_2(x)-1} \\
		g_1(x,s_1,s_2) &= \sin s_1 + \frac{1}{2}\cos s_2 + \rho_1(x) \\
		g_2(x,s_1,s_2) &= \frac{1}{2}\sin s_1 + \cos s_2 + \rho_2(x)
	\end{align*}
	fulfill \eqref{H4} (and hence also \eqref{H3}). Indeed, one can choose $h_1^{(n)} = \frac{\pi}{2}+2\pi n$, $k_1^{(n)} = \frac{3}{2}\pi+2\pi n$, $h_2^{(n)} = 2\pi n$, and $k_2^{(n)} = \pi+2\pi n$ for all $n\in\N$. Since $h_i^{(n)} \to +\infty$ as $n\to\infty$, the sequence of solutions given by Theorem \ref{Neumannsols} diverges at $+\infty$ a.\,e.\,uniformly in $\Omega$.
\end{example}

\section{The Dirichlet problem}\label{Section5}

In this section we want to discuss the situation when we have a Dirichlet boundary condition instead of a nonhomogeneous Neumann one. We consider the system
\begin{equation}\label{problem-dirichlet}
	\left\{
	\begin{aligned}
		-\divergenz \big(|\nabla u_1|^{p_1(x)-2}\nabla u_1 +\mu_1(x)|\nabla u_1|^{q_1(x)-2}\nabla u_1 \big)&=  f_1(x,u_1,u_2,\nabla u_1,\nabla u_2)\quad&& \text{in } \Omega,\\
		-\divergenz \big(|\nabla u_2|^{p_2(x)-2}\nabla u_2 +\mu_2(x)|\nabla u_2|^{q_2(x)-2}\nabla u_2 \big)&=  f_2(x,u_1,u_2,\nabla u_1,\nabla u_2)\quad&& \text{in } \Omega,\\
		u_1=u_2&=0 &&\text{on } \partial\Omega,
	\end{aligned}
	\right.
\end{equation}
where $p_i,q_i,\mu_i$, $i=1,2$ satisfy hypotheses \eqref{H1}. Instead of $\mathcal{W}$, we consider its subspace $\mathcal{W}_0 = W^{1,\mathcal{H}_1}_0 \times W^{1,\mathcal{H}_2}_0$ equipped with the norm induced by the one of $\mathcal{W}$.

\begin{definition}
	We say that $(u_1,u_2)\in\mathcal{W}_0$ is a weak solution to \eqref{problem-dirichlet} if
	\begin{align}\label{defsol1-dirichlet}
		\begin{split}
			&\into \l(|\nabla u_1|^{p_1(x)-2}\nabla u_1+\mu_1(x)|\nabla u_1|^{q_1(x)-2}\nabla u_1\r)\cdot \nabla v_1\,\diff x
			=\into f_1(x,u_1,u_2,\nabla u_1,\nabla u_2)v_1\,\diff x\\
		\end{split}
	\end{align}
	and
	\begin{align}\label{defsol2-dirichlet}
		\begin{split}
			&\into \l(|\nabla u_2|^{p_2(x)-2}\nabla u_2+\mu_2(x)|\nabla u_2|^{q_2(x)-2}\nabla u_2\r)\cdot \nabla v_2\,\diff x
			=\into f_2(x,u_1,u_2,\nabla u_1,\nabla u_2)v_2\,\diff x\\
		\end{split}
	\end{align}
	hold true for all $(v_1,v_2)\in\mathcal{W}_0$ and all the integrals in \eqref{defsol1-dirichlet} and \eqref{defsol2-dirichlet} are finite.
\end{definition}

The definition of a sub- and a supersolution of problem \eqref{problem-dirichlet} reads as follows.

\begin{definition}
	We say that $(\underline{u}_1,\underline{u}_2)$, $(\overline{u}_1, \overline{u}_2)\in \mathcal{W}$ form a pair of sub- and supersolution of problem \eqref{problem-dirichlet} if $\underline{u}_i\leq 0\leq \overline{u}_i$ a.\,e.\,in $\Omega$ for $i=1,2$ and
	\begin{align*}
			&\into \l(|\nabla \underline{u}_1|^{p_1(x)-2}\nabla \underline{u}_1+\mu_1(x)|\nabla \underline{u}_1|^{q_1(x)-2}\nabla \underline{u}_1\r)\cdot \nabla v_1\,\diff x
			-\into f_1(x,\underline{u}_1,w_2,\nabla \underline{u}_1,\nabla w_2)v_1\,\diff x\\
			&+\into \l(|\nabla \underline{u}_2|^{p_2(x)-2}\nabla \underline{u}_2+\mu_2(x)|\nabla \underline{u}_2|^{q_2(x)-2}\nabla \underline{u}_2\r)\cdot \nabla v_2\,\diff x
			-\into f_2(x,w_1,\underline{u}_2,\nabla w_1,\nabla \underline{u}_2)v_2\,\diff x \leq 0
	\end{align*}
	and
	\begin{align*}
			&\into \l(|\nabla \overline{u}_1|^{p_1(x)-2}\nabla \overline{u}_1+\mu_1(x)|\nabla \overline{u}_1|^{q_1(x)-2}\nabla \overline{u}_1\r)\cdot \nabla v_1\,\diff x
			-\into f_1(x,\overline{u}_1,w_2,\nabla \overline{u}_1,\nabla w_2)v_1\,\diff x\\
			&+\into \l(|\nabla \overline{u}_2|^{p_2(x)-2}\nabla \overline{u}_2+\mu_2(x)|\nabla \overline{u}_2|^{q_2(x)-2}\nabla \overline{u}_2\r)\cdot \nabla v_2\,\diff x
			-\into f_2(x,w_1,\overline{u}_2,\nabla w_1,\nabla \overline{u}_2)v_2\,\diff x \geq 0
	\end{align*}
	for all $(v_1,v_2)\in\mathcal{W}_0$, $v_1,v_2\geq 0$ a.\,e.\,in $\Omega$ and for all $(w_1,w_2)\in \mathcal{W}$ such that $\underline{u}_i\leq w_i\leq \overline{u}_i$ for $i=1,2$, with all integrals above to be finite.
\end{definition}

Adapting the proof of Theorem \ref{theorem-sub-supersolution} with slight modifications, we have the following result.

\begin{theorem}\label{theorem-sub-supersolution-dirichlet}
	Let hypotheses \eqref{H1} and \ref{H2i} be satisfied. If $[\underline{u},\overline{u}]$ is a trapping region of \eqref{problem-dirichlet}, then the system in \eqref{problem-dirichlet} has a solution $u\in\mathcal{W}_0\cap [\underline{u},\overline{u}]$. 
\end{theorem}

In order to construct a pair of sub- and supersolution, we suppose the following assumptions.

\begin{enumerate}[label=\textnormal{(H$5$)},ref=\textnormal{H$5$}]
	\item\label{H5}
	There exist $\hat{\ph}_i,\hat{\psi}_i\in L^{p_i'(\cdot)}(\Omega)$ such that $0 \leq \hat{\ph}_i \leq \hat{\psi}_i$, $\hat{\ph}_i \not\equiv 0$, and
	\begin{align*}
	\hat{\ph}_i(x) \leq f_i(x,s_1,s_2,\xi_1,\xi_2) \leq \hat{\psi}_i(x)
	\end{align*}
	for a.\,a.\,$x \in \Omega$ and for all $(s_1,s_2,\xi_1,\xi_2)\in[0,+\infty)\times[0,+\infty)\times\R^N\times\R^N$.
\end{enumerate}

\begin{theorem}\label{Dirichletsol}
	Under the hypotheses \eqref{H1} and \eqref{H5}, there exists $(u_1,u_2)\in \mathcal{W}_0$ solution to problem \eqref{problem-dirichlet}.
\end{theorem}

\begin{proof}
	Let us consider the auxiliary problems
	\begin{equation}\label{auxiliary-sub-dirichlet}
		\left\{
		\begin{aligned}
			-\divergenz \big(|\nabla u_1|^{p_1(x)-2}\nabla u_1 +\mu_1(x)|\nabla u_1|^{q_1(x)-2}\nabla u_1 \big)&=  \hat{\ph}_1(x)\quad&& \text{in } \Omega,\\
			-\divergenz \big(|\nabla u_2|^{p_2(x)-2}\nabla u_2 +\mu_2(x)|\nabla u_2|^{q_2(x)-2}\nabla u_2 \big)&=  \hat{\ph}_2(x)\quad&& \text{in } \Omega,\\
			u_1=u_2&=0 &&\text{on } \partial\Omega,
		\end{aligned}
		\right.
	\end{equation}
	and
	\begin{equation}\label{auxiliary-super-dirichlet}
		\left\{
		\begin{aligned}
			-\divergenz \big(|\nabla u_1|^{p_1(x)-2}\nabla u_1 +\mu_1(x)|\nabla u_1|^{q_1(x)-2}\nabla u_1 \big)&=  \hat{\psi}_1(x)\quad&& \text{in } \Omega,\\
			-\divergenz \big(|\nabla u_2|^{p_2(x)-2}\nabla u_2 +\mu_2(x)|\nabla u_2|^{q_2(x)-2}\nabla u_2 \big)&=  \hat{\psi}_2(x)\quad&& \text{in } \Omega,\\
			u_1=u_2&=0 &&\text{on } \partial\Omega.
		\end{aligned}
		\right.
	\end{equation}
	According to the Minty-Browder theorem (see, e.\,g., Corollary 6.1.34 in Papageorgiou-Winkert \cite{Papageorgiou-Winkert-2018}) and the embedding $W^{1,\mathcal{H}_i}_0(\Omega) \hookrightarrow L^{p_i(\cdot)}(\Omega)$ (see Proposition \ref{proposition_embeddings}(ii)), there exist solutions $\underline{u} = (\underline{u}_1,\underline{u}_2)\in\mathcal{W}_0$ and $\overline{u} = (\overline{u}_1,\overline{u}_2)\in\mathcal{W}_0$ of \eqref{auxiliary-sub-dirichlet} and \eqref{auxiliary-super-dirichlet}, respectively. Testing \eqref{auxiliary-sub-dirichlet} and \eqref{auxiliary-super-dirichlet} with $\underline{u}^-$ and recalling $\hat{\ph}_i \geq 0$, we see that $\underline{u}_i \geq 0$ a.\,e.\,in $\Omega$. In addition, $\hat{\ph}_i \not\equiv 0$ forces $\underline{u}_i \not\equiv 0$. Testing \eqref{auxiliary-sub-dirichlet} and \eqref{auxiliary-super-dirichlet} with $(\underline{u}-\overline{u})_+$, besides using $\hat{\ph}_i \leq \hat{\psi}_i$ and the strict monotonicity of the operators, yields $\underline{u}_i \leq \overline{u}_i$ a.\,e.\,in $\Omega$. Moreover, due to \eqref{H5}, $[\underline{u},\overline{u}]$ is a trapping region of \eqref{problem-dirichlet}. The conclusion thus follows by applying Theorem \ref{theorem-sub-supersolution-dirichlet}.
\end{proof}

\section*{Acknowledgment}

The first two authors are members of the Gruppo Nazionale  per l'Analisi Matematica, la Probabilit\`{a} e le loro Applicazioni  (GNAMPA) of the Istituto Nazionale di Alta Matematica (INdAM); they are supported by the research project PRIN 2017 `Nonlinear Differential Problems via Variational, Topological and Set-valued Methods' (Grant No. 2017AYM8XW) of MIUR. \\
The first author was also supported by the GNAMPA-INdAM Project CUP\_E55F22000270001 and the research project `MO.S.A.I.C.' PRA 2020--2022 `PIACERI' Linea 3 of the University of Catania. \\
The second author was also supported by the grant FFR 2021 `Roberto Livrea'. \\
The third author was financially supported by GNAMPA and thanks the University of Palermo for the kind hospitality during a research stay in May/June 2022.


\end{document}